\documentclass[12pt,letterpaper]{amsart}
\usepackage{geometry,xcolor,hyperref}
\usepackage{amsmath,amssymb,setspace}
\usepackage{mathpazo} 
\newtheorem{theorem}{Theorem}
\newtheorem{theoremA}{Theorem}

\newtheorem{proposition}{Proposition}
\newtheorem{lemma}{Lemma}
\newtheorem{corollary}{Corollary}
\theoremstyle{remark}
\newtheorem{remark}{Remark}
\newcommand{\C}{\mathbb{C}}
\newcommand{\D}{\Omega}
\newcommand{\ep}{\varepsilon}
\newcommand{\Dc}{\overline{\Omega}}
\newcommand{\dbar}{\overline{\partial}}
\newcommand{\zb}{\overline{z}}

\newcommand{\vphib}{\overline{\varphi}}

\onehalfspace

\title{Schatten class Hankel and $\overline{\partial}$-Neumann operators 
	on pseudoconvex domains in $\mathbb{C}^n$}

\author{N\.{i}hat G\"{o}khan G\"{o}\u{g}\"{u}\c{s}}
\address[N\.{i}hat G\"{o}khan G\"{o}\u{g}\"{u}\c{s}]{Sabanc\i{}
    University, Tuzla, 34956, Istanbul, Turkey}
\email{nggogus@sabanciuniv.edu}

\author{S\"{o}nmez \c{S}ahuto\u{g}lu}
\address[S\"{o}nmez \c{S}ahuto\u{g}lu]{University of Toledo,
    Department of Mathematics \& Statistics, Toledo, OH 43606, USA}
\curraddr{Sabanc\i{} University, Tuzla, 34956, Istanbul, Turkey}
\email{Sonmez.Sahutoglu@utoledo.edu}


\date{\today}
\begin{document}

\begin{abstract}
Let $\Omega$ be a  $C^2$-smooth bounded pseudoconvex domain
in $\mathbb{C}^n$ for $n\geq 2$ and let $\varphi$ be a holomorphic 
function on $\D$ that is $C^2$-smooth on the closure of  $\Omega$. 
We prove that if $H_{\overline{\varphi}}$ is in Schatten $p$-class for 
$p\leq 2n$ then $\varphi$ is a constant function. As a corollary, we 
show that the $\overline{\partial}$-Neumann operator on $\Omega$
is not Hilbert-Schmidt.
\end{abstract}

\maketitle

Let $\D$ be a bounded domain in $\C^n$ and $A^2(\D)$
denote the Bergman space, the set of square integrable
holomorphic functions on $\D$. We define the Hankel
operator $H_{\varphi}:A^2(\D)\to L^2(\D)$ with symbol
$\varphi\in L^{\infty}(\D)$ as follows: $H_{\varphi}f=(I-P)(\varphi f)$
for $f\in A^2(\D)$, where $I$ is the identity map and
$P:L^2(\D)\to A^2(\D)$ is the Bergman projection.

In this paper we study Schatten $p$-class membership of Hankel
operators. The Hankel operator $H_{\varphi}$ is said to be in the
Schatten $p$-class, $S_p$,  if the operator
$(H^*_{\varphi}H_{\varphi})^{p/2}$ is in the trace class, $S_1$. We 
recall that a self-adjoint compact operator on a separable Hilbert 
space is in $S_1$ if its eigenvalues are absolutely summable. We
note that $S_2$ is the class of Hilbert-Schmidt operators and we refer 
the reader to \cite{ZhuBook} for more information about these notions.

On the unit disc, $\mathbb{D}\subset \C$,
Arazy-Fisher-Peetre \cite{ArazyFisherPeetre88}
(see also \cite[Theorem 8.29]{ZhuBook})
showed that for $\varphi\in A^2(\mathbb{D})$ the
Hankel operator $H_{\vphib}$ is in the Schatten $p$-class
if and only if $\varphi$ is in the Besov space $B_p$ consisting of
holomorphic functions $\varphi$ on $\mathbb{D}$ such that
\[\int_{\mathbb{D}} |\varphi'(z)|^p(1-|z|^2)^{p-2}dV(z)<\infty\]
where $dV$ is the Lebesgue measure.

In higher dimensions, that is $\D\subset \C^n$ for $n\geq 2$, 
the first result is due to Kehe Zhu. He \cite{Zhu90} showed that 
in case $\D$ is the unit ball and $\varphi$ is holomorphic, 
$H_{\vphib}\in S_p$ for $p\leq 2n$ if and only if $\varphi$ is 
constant. Since then Schatten $p$-class membership of Hankel 
operators has been studied by many authors. For example, to 
list a few, it has been studied on the unit ball \cite{Zhu91,Xia02,Pau16},
strongly pseudoconvex domains \cite{Li93}, finite type pseudoconvex
domains in $\C^2$ \cite{KrantzLiRochberg97}, Reinhardt domains 
\cite{Le14,CelikZeytuncu13,CelikZeytuncu17}, and the  Fock spaces 
\cite{Schneider04,Schneider09,SeipYoussfi13}. In this paper,
we study it on $C^2$-smooth bounded pseudoconvex domains in
$\C^n$ for $n\geq 2$.  Throughout the paper $\mathcal{O}(\D)$
denotes the space of holomorphic functions on $\D$.

Our main result is the following theorem.
\begin{theorem}\label{ThmMain}
    Let $\D$ be a $C^2$-smooth bounded pseudoconvex domain in
    $\C^n$ for $n\geq 2$ and $\varphi\in \mathcal{O}(\D)\cap C^2(\Dc)$.
    Then $H_{\vphib}$ is in $S_{p}$ for $p\leq 2n$ if and only if
    $\varphi$ is a constant function.
\end{theorem}

The following is a trivial corollary of Theorem \ref{ThmMain}.
\begin{corollary}\label{CorMain}
    Let $\D$ be a $C^2$-smooth bounded pseudoconvex domain
    in $\C^n$ for $n\geq 2$ and  $\varphi\in \mathcal{O}(\D)\cap C^2(\Dc)$.
    Then $H_{\vphib}$ is Hilbert-Schmidt on the Bergman space $A^2(\D)$
    if and only if $\varphi$ is a constant function.
\end{corollary}

Hankel operators, through the Kohn's formula, are connected to
the $\dbar$-Neumann operator, an important tool in several
complex variables. Now we explain this connection.

Let $\dbar\dbar^*+\dbar^*\dbar$ be the complex Laplacian 
on $L^2_{(0,1)}(\D)$, the square integrable $(0,1)$-forms on $\D$. 
This is an unbounded, self-adjoint, closed operator. 
H\"{o}rmander \cite{Hormander65} showed that (see also 
\cite[Theorem 4.4.1]{ChenShawBook}), if $\D$ is a bounded 
pseudoconvex domain in $\C^n$, then the complex Laplacian 
has a bounded solution operator $N_1$, called the $\dbar$-Neumann 
operator. Furthermore, Kohn \cite{Kohn63} (see also 
\cite[Theorem 4.4.5]{ChenShawBook}) proved that the Bergman 
projection and $N_1$ are connected trough the following formula
\[P=I-\dbar^*N_1\dbar.\]
Therefore, one can show that if $\D$ is a bounded pseudoconvex
domain and  $\varphi\in C^1(\Dc)$ then
$H_{\varphi}f=\dbar^*N_1(f\dbar\varphi)$ for $f\in A^2(\D)$.
So it is reasonable to expect $H_{\varphi}$ to be closely connected
to $N_1$. Indeed this is true in terms of compactness of the operators.
We refer the reader to \cite[Proposition 4.1]{StraubeBook}
and  \cite{CuckovicSahutoglu09,CelikSahutoglu14,SahutogluZeytuncu17}
for some recent results in this direction, and to books \cite{ChenShawBook,StraubeBook,HaslingerBook}
for more information about the $\dbar$-Neumann problem. 

In terms of Schatten $p$-class membership of 
$N_1:L^2_{(0,1)}(\D)\to L^2_{(0,1)}(\D)$ we have the following corollary, 
which will be proven at the end of the paper. We note that the result 
in Corollary \ref{CorN} below also holds for the restriction of $N_1$ 
onto $A^2_{(0,1)}(\Omega )$,  the space of $(0,1)$-forms with square 
integrable holomorphic coefficients on $\D.$ Furthermore, while 
$\overline \partial^* N_1$ (canonical solution operator to $\dbar$)   
is Hilbert-Schmidt for $\D=\mathbb{D}\subset \mathbb{C}$, it 
fails to be Hilbert-Schmidt when $\D$ is the unit ball in $\mathbb C^n$ 
for $n\ge 2$.  We refer the reader to \cite[Chapter 2]{HaslingerBook} and 
the references therein for results about Schatten $p$-class membership of
$\dbar^*N_1$.

\begin{corollary}\label{CorN}
    Let $\D$ be a $C^2$-smooth bounded pseudoconvex domain in $\C^n$
    for $n\geq 2$ and $N_1$ denote the $\dbar$-Neumann operator.
    Then $\dbar^*N_1$ is not in $S_4$ and $N_1$ is not Hilbert-Schmidt.
\end{corollary}

The rest of the paper is organized as follows. In the next section we
will present some necessary basic results that are well known.
We include them here for the convenience of the reader. In the
last section we give the proofs of Theorem \ref{ThmMain} and
Corollary \ref{CorN}.

\section*{Preparatory Results}

In this section we will include some preparatory results
that will be useful in the proof of Theorem \ref{ThmMain}. 
We include them here for the convenience of the reader 
	but we don't claim any originality about these results.

Let $\D$ be a bounded domain and $\varphi \in L^{\infty}(\D)$.
Then the Berezin transform of $\varphi$ is defined as
\[ \widetilde{\varphi}(z)=\int_{\D}|k_z(\xi)|^2\varphi(\xi)dV(\xi)\]
where $k_z(\xi)=\frac{K(\xi,z)}{\sqrt{K(z,z)}}$. Furthermore, we define
\[MO(\varphi,z)=\widetilde{|\varphi|^2}(z)-|\widetilde{\varphi}(z)|^2.\]
We denote $H^{\infty}(\D)=\mathcal{O}(\D)\cap L^{\infty}(\D)$. In
case $\varphi\in H^{\infty}(\D)$ we have
\[MO(\varphi,z)=\widetilde{|\varphi|^2}(z)-|\varphi(z)|^2\]
as $\widetilde{\varphi}=\varphi$.

\begin{lemma}\label{LemBProj}
    Let $\D$ be a bounded domain in $\C^n$ and
    $\varphi\in H^{\infty}(\D)$. Then
    $P\vphib k_z=\overline{\varphi(z)}k_z$ for $z\in \D$.
\end{lemma}

\begin{proof}
        Let $z,w\in \D$. Then
\begin{align*}
P\vphib k_z(w)
=&\int_{\D}K(w,\xi)\overline{\varphi(\xi)}k_z(\xi)dV(\xi)\\
 =&\int_{\D}K(w,\xi)\frac{K(\xi,z)}{\sqrt{K(z,z)}}\overline{\varphi(\xi)}dV(\xi)\\
=&\frac{1}{\sqrt{K(z,z)}} \overline{\int_{\D}K(z,\xi)K(\xi,w)\varphi(\xi)dV(\xi)}\\
=&\frac{1}{\sqrt{K(z,z)}}\overline{K(z,w)\varphi(z)}\\
=&\overline{\varphi(z)}k_z(w).
\end{align*}
Hence the proof of the lemma is complete.
\end{proof}

\begin{corollary}
    Let $\D$ be a  bounded domain in $\C^n$ and
    $\varphi\in H^{\infty}(\D)$. Then
\[H_{\vphib}k_z(w)
    =(\overline{\varphi(w)}-\overline{\varphi(z)})k_z(w)\]
    for $z,w\in \D$.
\end{corollary}

\begin{lemma}\label{LemHankel}
Let $\D$ be a bounded domain in $\C^n$ and $\varphi\in H^{\infty}(\D)$.
Then $\|H_{\vphib}k_z\|^2=MO(\varphi,z)$.
\end{lemma}

\begin{proof}
    Let $z\in \D$. Lemma \ref{LemBProj} implies that
    $P\vphib k_z=\overline{\varphi(z)}k_z$. Then
\begin{align*}
\left\|H_{\vphib}k_z\right\|^2
=&\langle H_{\vphib}k_z,H_{\vphib}k_z\rangle\\
=&\langle \vphib k_z,\vphib k_z\rangle 
	-\langle P\vphib k_z,\vphib k_z\rangle\\
=&\langle |\vphib|^2 k_z, k_z\rangle-\overline{\varphi(z)}
    \langle  k_z,\vphib k_z\rangle\\
=&\widetilde{|\varphi|^2}(z)-|\varphi(z)|^2\\
=&MO(\varphi,z).
\end{align*}
    Hence the proof of the lemma is complete.
\end{proof}

We note that even though Lemmas  1 and 2 in \cite{Zhu91}
(used in the proof below) are stated for the ball, they are
actually true on any domain. The following corollary can
also be deduced from \cite[Theorem 3.1]{Li93}. We
present a proof here for the convenience of the reader.

\begin{corollary}\label{CorMO}
    Let $\D$ be a bounded domain in $\C^n, p\geq 2$,
    and $\varphi\in H^{\infty}(\D)$. Then $ H_{\vphib}\in S_p$
    implies that $\int_{\D}(MO(\varphi,z))^{p/2}K(z,z) dV(z)<\infty.$
\end{corollary}

\begin{proof}
    Let us assume that $H_{\vphib}\in S_p$ for $p\geq 2$. Then
    $(H_{\vphib}^*H_{\vphib})^{p/2}$ is in trace class on $A^2(\D)$
    (see \cite[Theorem 1.26]{ZhuBook}). Then \cite[Lemma 1]{Zhu91} 
    (see also proof of \cite[Theorem 6.4]{ZhuBook}) implies that
\[\int_{\D} \langle(H_{\vphib}^*H_{\vphib})^{p/2}k_z,k_z\rangle
    K(z,z)dV(z)<\infty.\]
Next we use Lemma \ref{LemHankel} and  \cite[Lemma 2]{Zhu91}
(see also \cite[Proposition 1.31]{ZhuBook}) to conclude that
\begin{align*}
 \int_{\D} (MO(\varphi,z))^{p/2} K(z,z)dV(z)
=& \int_{\D} \left\|H_{\vphib}k_z\right\|^p K(z,z)dV(z)\\
=& \int_{\D} \langle H_{\vphib}^*H_{\vphib}k_z,k_z\rangle^{p/2}
    K(z,z)dV(z)\\
\leq & \int_{\D} \langle(H_{\vphib}^*H_{\vphib})^{p/2}k_z,k_z\rangle
    K(z,z)dV(z)  \\
<&\infty.
\end{align*}
Therefore, the proof of the corollary is complete.
\end{proof}

\begin{remark}
We will use \cite[Theorem F]{BekolleBergerCoburnZhu90}
 in the proof of Theorem \ref{ThmMain}. So we take this
 opportunity to comment that even though 
 \cite[Theorem F]{BekolleBergerCoburnZhu90} is stated for
bounded symmetric domains, observation of the proof 
(see  \cite[Remark on pg 321]{BekolleBergerCoburnZhu90}
reveals that it is actually true on all bounded domains in $\C^n$.
Indeed, let $\psi:\C\to [0,\infty)$ be a rotation-invariant 
	$C^{\infty}$-smooth function with $supp(\psi)\subset \mathbb{D}$ 
	and $\int_{\mathbb{D}}\psi(\xi)dV(\xi)=1$. Then for $z\in \D$ 
	and sufficiently small $\ep>0$ we define  
\[\chi_z(w)=\frac{1}{\ep^{2n}}\psi\left(\frac{w_1-z_1}{\ep}\right)
	\cdots \psi\left(\frac{w_n-z_n}{\ep}\right)\in C^{\infty}_0(\D)\]
where $w=(w_1,\ldots, w_n)$ and $z=(z_1,\ldots,z_n)$. Then we 
have $K(w,z)=P\chi_z(w)$
(see, for instance, \cite[Remark 12.1.5]{JarnickiPflugBookEd2}).
To prove that   $\frac{\partial}{\partial \zb_j}K(.,z) \in A^2(\D)$,
it is enough to show that
\[\frac{\partial}{\partial x_j}P\chi_z
    = P\left(\frac{\partial}{\partial x_j}\chi_z\right) \text{ and }
        \frac{\partial}{\partial y_j}P\chi_z
        = P\left(\frac{\partial}{\partial y_j}\chi_z\right)\] 
 where $z_j=x_j+iy_j$. We will show only the first 
 equality as the second one is similar.
 Let $h_j=(0,\ldots, 0,h,0,\ldots, 0)$ where $h$ is a 
 real number at the $j$th spot. Since we are dealing with holomorphic
 functions, it is enough to prove that
 $\|P\chi_{z+h_j}-P\chi_z-hP\partial_{x_j}\chi_z\|=o(h)$ where
 $\partial_{x_j}=\frac{\partial}{\partial x_j}$. Since $P$ is a bounded linear
 operator with norm equal to 1 and $\chi_z\in C^{\infty}_0(\D)$ we have
\begin{align*}
        \left\| \frac{P\chi_{z+h_j}-P\chi_z-hP\partial_{x_j}\chi_z}{h}\right\|\leq
        \left\| \frac{\chi_{z+h_j}-\chi_z-h\partial_{x_j}\chi_z}{h}\right\|\to 0
\end{align*}
        as $h\to 0$. Therefore,
    $\frac{\partial}{\partial x_j}P\chi_z = P\left(\frac{\partial}{\partial x_j}\chi_z\right)$. 
    Furthermore, using induction we conclude that  
    $\frac{\partial^{\alpha}}{\partial \zb_j^{\alpha}}K(.,z) \in A^2(\D)$ for any 
    multi-index $\alpha$.
\end{remark}

    The following is a version of \cite[Theorem F]{BekolleBergerCoburnZhu90} 
    for bounded domains in $\C^n$.

\begin{theorem}[\cite{BekolleBergerCoburnZhu90}]
    Let $\D$ be a bounded domain in $\C^n$ and $\gamma:[0,1]\to \D$
    be a $C^1$-smooth curve. Assume that $s(t)$ denote the arc-length
    of $\gamma$ with respect to the Bergman metric of $\D$ and
    $\varphi \in L^{\infty}(\D)$. Then
\[\left| \frac{d}{dt}\widetilde{\varphi}(\gamma(t))\right|
    \leq 2\sqrt{2} \left(\frac{ds}{dt}\right)
    \sup_{0\leq t\leq 1}(MO(\varphi,\gamma(t)))^{1/2}.\]
\end{theorem}

Then we have the following useful corollary.

\begin{corollary}\label{CorZhu}
Let $\D$ be a bounded domain in $\C^n, \varphi \in H^{\infty}(\D)$,
and $X=(a_1,\ldots,a_n)\in \C^n$. Then
\[\left| \sum_{j=1}^n a_j\frac{\partial \varphi(z)}{\partial z_j}\right|
\leq 2\sqrt{2} (MO(\varphi,z))^{1/2} B(X,z)\]
where $B(X,z)$ denotes the Bergman metric applied to the
vector $X$ at $z$.
\end{corollary}

\section*{Proofs  of Theorem \ref{ThmMain} and Corollary \ref{CorN}}
Before we start the proof of Theorem \ref{ThmMain} we present two
results in several complex variables. We note that $B_{z_0}(r)$ denotes 
the open ball centered at $z_0$ with radius $r$.  We will use the notion 
of CR functions in the following proposition. We refer the reader to 
\cite[Chapter 3]{ChenShawBook} for the definition and properties 
of CR functions. 

\begin{proposition}\label{PropCR}
    Let $\D$ be a domain in $\C^n$ for $n\geq 2, z_0\in b\D$,
    and $\varphi\in \mathcal{O}(\D)\cap C^2(\Dc)$. Assume that
     there exists $r>0$ such that $b\D$ is $C^2$-smooth in the ball 
     $B_{z_0}(r)$, the Levi form of $b\D$ has at least one positive 
     eigenvalue at $z_0$, and $\vphib$ is CR function on 
     $b\D\cap B_{z_0}(r)$. Then $\varphi$ is constant.
\end{proposition}
\begin{proof}
    Using a holomorphic change of coordinates we may assume
    that $z_0$ is the origin, $y_n$-axis is the real normal direction
    and $X_1=(0,\ldots, 0,1,0)$ is complex tangential
    (corresponding to a positive eigenvalue of the Levi form,
    and the two dimensional slice $H_0$) at $z_0$, and
    \[H_0=\{(\xi_1,\xi_2)\in \C^2:(0,\ldots, 0,\xi_1,\xi_2)\in \D\}\]
    is strictly convex at the origin. Furthermore, since small 
    $C^2$ perturbations of strictly convex surfaces
    are strictly convex, the slices
    $\{(\xi_1,\xi_2)\in \C^2:(z_1,\ldots, z_{n-2},\xi_1,\xi_2)\in \D\}$
    are strictly convex for sufficiently small $|z_1|+\ldots+|z_{n-2}|$.
    Then we conclude that there exists $0<c<1$ such that
    $\D\cap B_{z_0}(cr)$ is union of discs parallel to $z_1$-axis
    whose boundaries lie in $b\D\cap B_{z_0}(r)$.

     Since $\vphib|_{b\D\cap B_{z_0}(r)}$ is a CR function,
     \cite[Theorem 3.3.2]{ChenShawBook} implies that
     it has a holomorphic extension $\phi_{z_0,r}$ onto
     $\D\cap B_{z_0}(cr)$ for some $c>0$ (here we shrink
     $c$ if necessary). Then the fact that $\phi_{z_0,r}$ and $\vphib$
    are harmonic and they match on  $b\D\cap B_{z_0}(r)$ imply that
    $\phi_{z_0,r}=\vphib$ on $\D\cap B_{z_0}(cr)$.
    Hence, $\varphi$ and $\vphib$ are holomorphic on
    $\D\cap B_{z_0}(cr)$. Therefore, $\varphi$ is constant.
\end{proof}

    In the following theorem (see also \cite[Theorem 6.8]{OhsawaBook1}
    for a statement) $\pi(z)$ denotes the point in $b\D$ closest to $z$ and
    $d_{b\D}(z)$ denotes the distance from $z$ to $b\D$. We note that
    the function $\pi$ is well defined near $C^2$-smooth portion of
    the boundary.

\begin{theorem}[Diederich \cite{Diederich70}]\label{ThmDiederich}
    Let $\D$ be a pseudoconvex domain in $\C^n$ and $z_0\in b\D$.
    Assume that there exists an open neighborhood $U$ of $z_0$
    such that $b\D$ is $C^2$-smooth in $U$ and $b\D\cap U$ is composed
    of strongly pseudoconvex points. Then there exists a neighborhood
    $V\Subset U$ of $z_0$ and $C>0$ such that
\[B(X,z) \leq C
    \left(\frac{|X_{\tau}|}{(d_{b\D}(z))^{1/2}}
    +\frac{|X_\nu|}{d_{b\D}(z)}\right)\]
    for $z\in V\cap \D$ where $X_{\tau}$ and $X_{\nu}$ denote that complex
    tangential and complex normal component of $X$ at $\pi(z)$, respectively.
\end{theorem}

Now we are ready to present the proof of Theorem \ref{ThmMain}.
We will use the fact that every bounded $C^2$-smooth pseudoconvex
domain has some strongly pseudoconvex boundary points (see,
for instance, \cite{Basener77}). Then we will follow the ideas in
\cite{Li93} and localize the estimate near a strongly pseudoconvex
point in the boundary to get a contradiction in case
$H_{\vphib}\in S_{p}$ for $p\leq 2n$.

\begin{proof}[Proof of Theorem \ref{ThmMain}]
	We will only prove the non-trivial direction. 
    Since $S_{\alpha}\subseteq S_{\beta}$ for $\alpha\leq \beta$
    we start the proof by assuming that $ H_{\vphib}\in S_{2n}$.
    Then Corollary \ref{CorMO} (see also \cite[Theorem 3.1]{Li93})
    implies that
\begin{align}\label{Eqn1}
\int_{\D}(MO(\varphi,z))^n K(z,z)dV(z)<\infty.
\end{align}
    Let $z_0\in b\D$ be a strongly pseudoconvex point and $U=B_{z_0}(r)$
    so that all points in $B_{z_0}(2r)\cap b\D$ are strongly pseudoconvex.
     By Corollary \ref{CorZhu} we have
\[\left| \sum_{j=1}^n a_j\frac{\partial \varphi(z)}{\partial z_j}\right|
    \leq 2\sqrt{2} (MO(\varphi,z))^{1/2} B(X,z)\]
    for $X=(a_1,\ldots,a_n)\in \C^n$. Furthermore,
    Theorem \ref{ThmDiederich}  implies that there  exists $C>0$ such that
\[B(X,z) \leq \frac{C}{2\sqrt{2}}
    \left(\frac{|X_\tau|}{(d_{b\D}(z))^{1/2}}+\frac{|X_\nu|}{d_{b\D}(z)}\right)\]
    for $z\in U\cap \D$ where $X_{\tau}$ and $X_{\nu}$ are the tangential
    and normal components of $X$, respectively. Combining the previous
    two estimates, we conclude that for any $z\in U\cap \D$ we have
\[\left| \sum_{j=1}^n a_j\frac{\partial \varphi(z)}{\partial z_j}\right|
    \leq C\cdot  (MO(\varphi,z))^{1/2}
    \left(\frac{|X_\tau|}{(d_{b\D}(z))^{1/2}}
    +\frac{|X_\nu|}{d_{b\D}(z)}\right).\]
    Then
\[|\partial_b \varphi(z)|^2d_{b\D}(z)
    =|\dbar_b \vphib(z)|^2d_{b\D}(z)
    \leq C^2\cdot  MO(\varphi,z). \]
    Combining the previous inequality with \eqref{Eqn1} we get
\[\int_{\D\cap U}
    |\partial_b \varphi(z)|^{2n}(d_{b\D}(z))^n K(z,z)dV(z)
    <\infty.\]
    We note that  $K(z,z)$ is comparable to $(d_{b\D}(z))^{-n-1}$ near
    strongly pseudoconvex boundary points (see, for example,
    \cite[Theorem 3.5.1]{Hormander65}). Then there exists
    $\widetilde{C}>0$ such that for sufficiently small $\ep>0$ we get
\begin{align*}
\int_0^{\ep}\frac{dt}{t}\int_{b\D_t\cap \widetilde{U}}
        |\partial_b \varphi(z)|^{2n}d\sigma(z)
\leq& \widetilde{C} \int_{\D\cap U}
        \frac{|\partial_b \varphi(z)|^{2n}}{d_{b\D}(z)}dV(z)\\
\leq &\widetilde{C}^2 \int_{\D\cap U}
        |\partial_b \varphi(z)|^{2n}(d_{b\D}(z))^n K(z,z)dV(z)\\
<&\infty
\end{align*}
    where $b\D_t=\{z\in \D:d_{b\D}(z)=t\}$ and $\widetilde{U}=B_{z_0}(r/2)$.
    Then $\int_{b\D\cap U} |\partial_b \varphi(z)|^{2n}d\sigma(z)=0$.
    Since  $\partial_b \varphi$ is continuous  on $b\D\cap U$
    we conclude that $\partial_b\varphi=0$ on $b\D\cap U$.
    Finally, Proposition \ref{PropCR} implies that $\varphi$ is constant.
\end{proof}

Finally we present the proof of Corollary \ref{CorN}.
\begin{proof}[Proof of Corollary \ref{CorN}]
    Let $K^2_{(0,q)}(\D)$ denote the square integrable $\dbar$-closed
    $(0,q)$-forms on $\D$ and $N_q$ denote the $\dbar$-Neumann
    operator on $L^2_{(0,q)}(\D)$. We note that $K^2_{(0,1)}(\D)$ is
    a closed subspace (as it is the kernel of $\dbar$) of  $L^2_{(0,1)}(\D)$
    and $N_1$ maps $K^2_{(0,1)}(\D)$ into itself (as $\dbar N_1=N_2\dbar$).
    Range's Theorem (see, for instance, \cite[p.77]{StraubeBook}
    and \cite{Range84})  implies that $N_1=(\dbar^*N_1)^*\dbar^*N_1$ on
    $Ker(\dbar)$. Furthermore,  $T\in S_p$ if and only if $T^*T\in S_{p/2}$
    (see \cite[Theorem 1.26]{ZhuBook}). If $N_1$ is Hilbert-Schmidt
    then $\dbar^*N_1|_{A^2_{(0,1)}(\D)}\in S_4$ where $A^2_{(0,1)}(\D)$
    is the space of $(0,1)$-forms with square integrable holomorphic
    coefficients. However, $H_{\zb_1}f=\dbar^*N_1(fd\zb_1)$
    for $f\in A^2(\D)$ and $H_{\zb_1}\not\in S_4$. Therefore,
    $\dbar^*N_1\not\in S_4$ and $N_1$ is not Hilbert-Schmidt.
\end{proof}

\section*{Acknowledgment}
Part of this work was done while the second author was visiting Sabanc\i{}
University. He thanks this institution for its hospitality and good working
conditions.  He also thanks Trieu Le for fruitful discussions. We are thankful 
to the anonymous referee for constructive comments that improved the 
presentation of the paper.

\end{document}